\documentclass{amsart}
\usepackage{amsfonts}
\usepackage{color}

\usepackage{mathrsfs}

\usepackage{graphicx}
\usepackage{nicefrac}
\usepackage{mathtools}
\usepackage{amssymb}

\newtheorem{thm}{Theorem}[section]
\newtheorem{cor}[thm]{Corollary}
\newtheorem{lema}[thm]{Lemma}
\newtheorem{prop}[thm]{Proposition}
\theoremstyle{definition}

\newtheorem{rem}[thm]{Remark}

\numberwithin{equation}{section}
\newcommand{\R}{\mathbb R}
\newcommand{\N}{\mathbb N}

\newcommand{\W}{\mathcal{W}^{s,p}}
\newcommand{\ve}{\varepsilon}
\newcommand{\lam}{\lambda}
\newcommand{\lf}{(-\Delta)^s_p}
\newcommand{\nn}{\mathcal{N}_{s,p}}
\newcommand{\hh}{\mathcal{H}_{s,p}}
\newcommand{\C}{\mathcal{K}_{n,p}}
\newcommand{\cd}{\rightharpoonup}

\newcommand{\lp}{L^{p}(\Omega)}
\newcommand{\wsp}{W^{s,p}(\Omega)}

\begin{document}
\title[Fractional eigenvalue problems]{Fractional eigenvalue problems that approximate Steklov eigenvalues}
\author[L. M. Del Pezzo, J. D. Rossi, A. M. Salort]
{Leandro M. Del Pezzo, Julio D. Rossi,  Ariel M. Salort}
\address{Departamento de Matem\'atica
 \hfill\break \indent FCEN - Universidad de Buenos Aires and
 \hfill\break \indent   IMAS - CONICET.
\hfill\break \indent Ciudad Universitaria, Pabell\'on I \hfill\break \indent   (1428)
Av. Cantilo s/n. \hfill\break \indent Buenos Aires, Argentina.}
\email[A.M. Salort]{asalort@dm.uba.ar}
\email[L. M. Del Pezzo]{ldpezzo@dm.uba.ar}
\email[J. D. Rossi]{jrossi@dm.uba.ar}



\begin{abstract}
	In this paper we analyze possible extensions of the classical Steklov eigenvalue problem
	to the fractional setting. In particular, we find
	a nonlocal eigenvalue problem of fractional 
	type that approximate, when taking a suitable limit, 
	the classical Steklov eigenvalue problem.
\end{abstract}

\maketitle
\section{Introduction}
	
	Of crucial importance in the study of boundary value problems for
	differential operators are the Sobolev spaces and
	inequalities. Hence, the Sobolev inequalities and their optimal
	constants is a subject of interest in the analysis of PDE's and
	related topics. They have been widely studied in the past by many
	authors and is still an area of intensive research, see the book 
	\cite{A} and the survey \cite{DH} for an introduction to this field.

	When analyzing elliptic or parabolic problems with 
	nonlinear boundary conditions it turns out that among the Sobolev 
	embeddings, a fundamental role is played by the Sobolev trace 
	theorem. The study of the best constant in the Sobolev 
	trace theorem leads 
	naturally to eigenvalue problems known in the literature as
	Steklov eigenvalues.
	
	\medskip
		
	Our main goal in this paper is to analyze a fractional
	approximation for Steklov eigenvalues.
	Given a bounded domain $\Omega\subset \R^n,$ $s\in(0,1)$
	and $p\in(1,\infty),$  we are aimed at studying the following nonlocal problem
	\begin{align} \label{ecu}
		\begin{cases}
		  \C(1-s)\lf u + |u|^{p-2}u=\dfrac{\lam}\ve\chi_{ \Omega_\ve}|u|^{p-2}u 
		  &\quad \mbox{ in } \Omega,  \\
	  	  \nn u=0 &\quad\mbox{ in }  \Omega^c=\R^n \setminus \overline\Omega,
		\end{cases}
	\end{align}
	where $s$ and $\ve$ are real numbers belonging to $(0,1)$ and  $\Omega_\ve\coloneqq \{x\in \Omega \colon d(x,\Omega)\leq \ve\}$. 
	The fractional $p-$Laplacian is defined as
		$$
		  \lf u (x)=  2\,\mbox{p.v.} \int_{\R^n} \frac{|u(x)-u(y)|^{p-2}(u(x)-u(y))}{|x-y|^{n+sp}}\, dy,
		$$
	and $\nn$ is the associated nonlocal derivative defined in \cite{DROV}  by
	\begin{equation} \label{deriv}
		\nn u (x)\coloneqq 2\, \int_\Omega \frac{|u(x)-u(y)|^{p-2}(u(x)-u(y))}{|x-y|^{n+sp}} \, dy, \qquad 
		x\in \R^n \setminus \overline\Omega.
	\end{equation}
	The constant $\C$ is the normalization constant computed in \cite{BBM}. In fact,
	although the fractional seminorm $[u]_{s,p}\to \infty$ as $s\to 1^-$, Bourgain, Brezis and Mironescu in \cite{BBM} proved that for any smooth bounded domain $\Omega\subset \R^n$, $u\in W^{1,p}(\Omega)$ with $p\in(1,\infty)$ there exists a constant $\C$ such that
	\begin{equation} \label{lim.grad}
		\lim_{s\to 1^-}\C (1-s)\iint_{ \Omega \times \Omega} \frac{|u(x)-u(y)|^p}{|x-y|^{n+sp}}dx dy=\int_\Omega |\nabla u|^p\,dx.
	\end{equation}
	The constant can be explicitly computed and is given by 
	\[
		\C=\frac{p\Gamma(\frac{n+p}{2})}
		{2\pi^\frac{n-1}{2}\Gamma(\frac{p+1}{2})}.
	\]
	
	As the authors of \cite{DROV} pointed out, one of the main advantages in using this form of nonlocal derivative arises in the following nonlocal divergence theorem: for any bounded smooth enough functions 
	$u$ and $v$  it holds that
	\begin{equation} \label{teo.div}
		\int_\Omega \lf u (x) \, dx = - 
		\int_{\Omega^c} \nn u(x)\, dx.
	\end{equation}
	Moreover, the following integration by parts formula is true
	\begin{equation} \label{partes}
		  \hh (u,v)=\int_\Omega v(x) \lf u(x)
		\, dx + \int_{\Omega^c} v(x) \nn u(x)\, dx,
	\end{equation}	
	 where 
	$$
		\hh(u,v):= 
		\iint_{\R^{2n} \setminus (\Omega^c)^2} 
		\frac{|u(x)-u(y)|^{p-2}(u(x)-u(y))(v(x)-v(y))}{|x-y|^{n+sp}} 
		\, dx\, dy.
	$$
	
	By multiplying \eqref{ecu} by bounded smooth enough function  
	$v$, integrating in $ \Omega$ and by using \eqref{partes} we obtain 
	the following weak formulation for \eqref{ecu}
	\begin{equation} \label{debil}
		\C(1-s) \hh(u,v)   + \int_{\Omega} |u|^{p-2}uv
		\, dx = \dfrac{\lam}{\ve} 
		\int_{ \Omega_\ve} |u|^{p-2}uv\, dx.
	\end{equation}
	We introduce some notation we will use along the paper.	Given a measurable function $u:\R^n \to \R$ we set
	$$
		\|u\|_{s,p}:=(\|u\|_{L^p(\Omega)}^p +[u]_{s,p}^p )^\frac{1}{p}, \qquad \mbox{ where} \quad [u]_{s,p}:= \left(\hh(u,u) \right)^\frac{1}{p}.
	$$	
	Associated with this norm the natural space to consider is the following 
	$$
		\W(\Omega) := \{ u \colon \mathbb{R}^n\to \mathbb{R}
		\text{ measurable } \colon \, \|u\|_{s,p}<
		\infty\}.
	$$
	For a fixed value $\ve>0,$ we say that the value 
	$\lambda\in \mathbb{R}$ is an eigenvalue of problem 
	\eqref{ecu} if there is $u\in \W(\Omega)$
	such that \eqref{debil} holds	for any $v\in \W(\Omega)$.	Note that if $\lambda\in \mathbb{R}$ is an eigenvalue of problem 
	\eqref{ecu} and $u$ is an associated eigenfunction, 
	then $\lambda>0$ and $u\not \equiv0$ in $\Omega_\ve.$ 
	Thus the first eigenvalue of \eqref{ecu} is given by
	\begin{equation} \label{lam1}
			\lam_{1,\ve}(s,p)=
			\inf_{\substack{u\in \W(\Omega),\\ 
			\|u\|^p_{L^p(\Omega_\ve )}\neq0} }
			\frac{\C(1-s)[u]_{s,p}^p + 
			\|u\|^p_{L^p(\Omega)} }{ \dfrac{1}{\ve}
			\|u\|^p_{L^p( \Omega_\ve ) }}.
	\end{equation}
	
	 Recall that it is well-known that the first eigenvalue of the Steklov problem
	\begin{align} \label{ecu2}
	\begin{cases}
		-\Delta_p u + |u|^{p-2}u = 0 &\mbox{ in } \Omega,\\
		|\nabla u|^{p-2} \frac{\partial u}{\partial \nu} = \lam |u|^{p-2}u &\mbox{ on } \partial \Omega,
	\end{cases}
	\end{align}
	is given by
	\begin{equation} \label{lam2}
		\lam_1(p)=\inf_{\substack{u\in W^{1,p}(\Omega),\\ 
		u\neq0} } \frac{\|\nabla u\|_{L^p(\Omega)}^p+ 
		\|u\|_{L^p(\Omega)}^p}{ \|u\|_{L^p(\partial\Omega)}^p }.
	\end{equation}
	Here the $p-$Laplacian is defined as $\Delta_p u = div(|\nabla 
	u|^{p-2} \nabla u)$ for $p\in (1,\infty)$.
	
	Taking $\epsilon = 1-s$, we are interested in studying the behavior of 
	$\lam_{1,1-s}(s,p)$ as $s \to 1^-$. Intuitively, a connection between 
	the limit of such eigenvalue and  $\lam_1(p)$, the first eigenvalue of 
	the Steklov $p-$Laplacian in $\Omega$, is expected to be found. 
	Indeed, note that from \eqref{lim.grad} one has that, for a fixed $u$,
	$$
		\lim_{s\to 1^{-}} \C(1-s)[u]_{s,p\,\Omega}^p + 
			\|u\|^p_{L^p(\Omega)}  = \|\nabla u\|_{L^p(\Omega)}^p+ 
			\|u\|_{L^p(\Omega)}^p;
	$$	
	and, moreover, since $\Omega_\ve\coloneqq \{x\in \Omega \colon 
	d(x,\Omega)\leq \ve\}$ is a strip around the boundary $\partial 
	\Omega$ of size $|\Omega_\ve | \sim \ve \times |\partial \Omega|$ one 
	expects that
	$$
		\lim_{s\to 1^-} \dfrac{1}{1-s} \int_{\Omega_{1-s}}
		|u|^p \, dx = \int_{\partial \Omega} |u|^p d\sigma.
	$$
	Note that the choice $\ve = 1-s$ is precise for this limit to hold.

	Our main results can be summarized as follows.

	\begin{thm}  \label{teo.1}
		There exists a
		sequence of eigenvalues of \eqref{ecu} $\lambda_{k,\epsilon} 
		(s,p)$ such that $\lambda_{k,\epsilon} (s,p)\to+\infty$ 
		as $k\to+\infty.$
 Every eigenfunction of \eqref{ecu} is in $L^{\infty}(\Omega).$
		
		The first eigenvalue $\lam_{1,\epsilon} (s,p)$ of \eqref{ecu} is 
		isolated and simple and has eigenfunctions that do not change sign.
		
		Moreover, choosing $\epsilon = 1-s$, we have the convergence of the first eigenvalue to the first Steklov eigenvalue as $s\to 1^-$, that is,
		$$
			\lim_{s\to 1^-} \lam_{1,1-s} (s,p)= \lam_1(p).
		$$	
	\end{thm}

	\begin{rem}
		It seems natural to consider 
		\begin{align} \label{ecu.33}
			\begin{cases}
				  \C(1-s)\lf u + |u|^{p-2}u=0 
				  &\quad \mbox{ in } \Omega,  \\
			  	  \nn u=\lambda |u|^{p-2}u &\quad\mbox{ in }  \Omega^c.
			\end{cases}
		\end{align}
		Associated with the first eigenvalue in this problem one has the 
		following minimization problem
		\begin{equation} \label{lam1.22}
			\tilde{\lam}_1(s,p)=
			\inf_{\substack{u\in \W(\Omega),\\ 
			\|u\|^p_{L^p(\Omega^c )}\neq0} }
			\frac{\C(1-s)[u]_{s,p\,\Omega}^p + 
			\|u\|^p_{L^p(\Omega)} }{ 
			\|u\|^p_{L^p( \Omega^c ) }}.
		\end{equation}
		However, this idea gives 
		$$
			\tilde{\lam}_1(s,p) =0
		$$	
		as can be easily obtained just by considering as a minimizing 
		sequence $u_k (x) = \phi (x+ ke_1)$ with $\phi $ a $C^\infty$ 
		compactly supported profile. 
	\end{rem}
	
	\begin{rem} \label{rem.trezas} {\rm
	When a trace embedding theorem holds (that is, for $ps>1$) we can consider the best fractional Sobolev trace constant that is given by
	\begin{equation} \label{trazas}
	\Lambda_1 (s,p) = \inf_{\begin{array}{c}
	u\in W^{s,p} (\Omega) \\
	u|_{\partial \Omega} \not\equiv 0 
	\end{array}} \dfrac{\displaystyle \C (1-s)\iint_{ \Omega \times \Omega} \frac{|u(x)-u(y)|^p}{|x-y|^{n+sp}}dxdy + \|u\|_{L^p(\Omega)}^p }
	{\displaystyle \int_{\partial\Omega} |u|^p\,d\sigma}.
	\end{equation}
	Thanks to the compactness of the embedding $W^{s,p} (\Omega) \hookrightarrow L^p (\partial \Omega)$ this infimum is attained and the minimizers are solutions to
	$$
	\begin{array}{l}
	\displaystyle
	\C(1-s) \iint_{\Omega\times \Omega} \frac{|u(x)-u(y)|^{p-2} (u(x)-u(y)) (v(x)-v(y))}{|x-y|^{n+sp}}dxdy  \\[10pt]
		\qquad \displaystyle  + \int_{\Omega} |u|^{p-2}uv
		\, dx = \Lambda_1(s,p) 
		\int_{\partial \Omega} |u|^{p-2}uv\, dx,
		\end{array}
	$$
for every $v\in W^{s,p} (\Omega)$.
Note that with this formulation it is not clear how to identify the ``boundary condition" satisfied 
by a minimizer $u$ (the equation inside the domain reads as 
$$ 
\C(1-s)\int_\Omega 
\frac{|u(x)-u(y)|^{p-2} (u(x)-u(y))}{|x-y|^{n+sp}}dy + |u|^{p-2}u (x) =0 
$$ for $x\in \Omega$). This is why we choose to analyze
\eqref{lam1} (that has \eqref{ecu} as associated PDE problem) instead of \eqref{trazas}.

With the same ideas used in the study of the limit as $s\to 1^-$ in Theorem \ref{teo.1} (see Section \ref{sect-limite}) one can show that
$$
\lim_{s\to 1^-} \Lambda_{1} (s,p)= \lam_1(p).
$$
We leave the details to the reader.}
\end{rem}

The paper is organized as follows: In Section \ref{Prel} we gather some preliminary results, in particular we show a minimum principle for our problem; in Section \ref{primer.atuvalor} we deal with the eigenvalue problem \eqref{ecu} and prove the first part of Theorem \ref{teo.1}; finally, in Section \ref{sect-limite} we analyze the limit
as $s\to 1^-$.

\section*{Acknowledgements}
We want to thank Ricardo Duran for several interesting discussions.
\section{Preliminaries}\label{Prel}

	We denote the usual fractional Sobolev spaces by $W^{s,p}(\Omega)$ for 
	$p\in [1,\infty)$ and $s\in(0,1)$   endowed with the norm
	\[
		\|u\|^p_{W^{s,p}(\Omega)}\coloneqq
		  \|u\|_{L^p(\Omega)}^p + 
		  \iint_{\Omega^2} \frac{|u(x)-u(y)|^p}{|x-y|^{n+sp}} \, 
		  dx \,dy.
	\]
	In the following,  $|u|_{\wsp}$ denotes usual Gagliardo seminorm 
	defined as
	\[
		|u|_{W^{s,p}(\Omega)}\coloneqq 
			\left(\displaystyle\iint_{\Omega^2} 
			\dfrac{|u(x)-u(y)|^p}{|x-y|^{n+sp}} \, dx \,dy\right)^{\frac1p} 
				\]
				for $1\le p<\infty$.
	It is easy to check that $\W(\Omega)$  is a subset of  
	$W^{s,p}(\Omega)$ for all $s\in (0,1)$.
	
	It will be quite useful here to establish the fractional compact 
	embeddings. For the proof see \cite{DD}.
	\begin{thm}\label{teo:inclucomp} 
		Let $\Omega\subset\mathbb{R}^n$
		a bounded open set with Lipschitz boundary, $s\in(0,1)$ and 
		$p\in(1,\infty).$ Then we have the following compact embeddings:
		\begin{align*}
			&W^{s,p}(\Omega)\hookrightarrow L^q(\Omega) 
			\qquad \mbox{ for all } q\in[1,p_s^\star), &\mbox{ if } sp\le n;\\
			&W^{s,p}(\Omega)\hookrightarrow C^{0,\lambda}_b(\Omega) 
			\quad \mbox{ for all } 
			\lambda<s-\nicefrac{n}{p}, &\mbox{ if } sp>n.
		\end{align*}
		Where $p_s^\star$ is the fractional critical Sobolev exponent, 
		that is
		\[
			p_s^\star\coloneqq
			\begin{cases}
			\dfrac{np}{n-sp} &\text{ if } sp<n,\\
			\infty  &\text{ if } sp\ge n.\\
			\end{cases}
		\]
	\end{thm}

	\subsection{A minimum principle}
		Here, we follow the ideas in \cite{BF}. 

		Given $s,\ve\in(0,1)$ and $p\in(1,\infty).$ We say that  
		$u\in\W(\Omega)$ is a week super-solution of 
		\begin{equation}\label{eq:a1}
			\begin{cases}
				\C(1-s)(-\Delta)_p^s u+|u|^{p-2}u=0 &\text{ in }\Omega,\\
				\nn u=0 &\text{ in }\Omega^{c}.
			\end{cases}
		\end{equation}
		iff
		\begin{equation}\label{eq:a2}
			\C(1-s) \hh(u,v)+\int_{\Omega} |u|^{p-2}u v\, dx\ge0	
		\end{equation}
		for every $v\in\W(\Omega), v\ge0.$

		First we need a subtle adaptation of 
		Lemma 1.3 in \cite{DKP}.
	
		\begin{lema}\label{lema:DKP}
			Let $s,\ve\in (0,1)$ and $p\in(1,\infty).$ Suppose that $u$ is a weak super-solution of  
			\eqref{eq:a1}  and $u\geq 0$ 
			in $\mathbb{R}^n$. If 
			$B_R(x_0)\subset\R^n\setminus\partial\Omega$ then  
			for any $B_r=B_r(x_0)\subset  B_{\nicefrac{R}2}(x_0) 
			$ and $0<\delta<1$
			\[
  				\iint_{B_r\times A}  \dfrac{1}{|x-y|^{n+sp}} \left| 
  				\log \left(\dfrac{u(x)+\delta}{u(y)+\delta}\right)
  				\right|^p \, dxdy \le 
   				 Cr^{n-sp}(1+r^{sp}),
			\]
			where 
			\[
				A=
				\begin{cases}
					B_r &\mbox{if }B_{R}\subset \Omega,\\
					\Omega &\mbox{if }B_{R}\subset 
					\R^n\setminus\overline{\Omega},\\
				\end{cases}
			\]
			and $C$ is a constant independent on $\delta$.
	\end{lema}
	
	\begin{proof}
		Let $0<r<\nicefrac{R}{2},$ $0<\delta$ and 
		$\phi\in C_0^\infty( B_{\nicefrac{3r}2})$ be such that 	
		\[
	  		0\le \phi \le 1, \quad \phi\equiv1 \text{ in } B_r \quad \text
	  		{ and } \quad|D\phi|<Cr^{-1} \text{ in } 
	  		B_{\nicefrac{3r}2}\subset B_{R}. 
		\]
		Taking $v=(u+\delta)^{1-p}\phi^p$ as test function in 
		\eqref{eq:a2} we have that
		\begin{equation}\label{eq:DKP1}
  			0 \le \C(1-s)\hh (u,(u+\delta)^{1-p}\phi^p)
    		+\int_{B_{\nicefrac{3r}2}\cap\Omega}
    		\frac{u^{p-1}}{(u+\delta)^{p-1}} \phi^p\, 			
    		dx . 
		\end{equation}
	
		On the other hand, in the proof of Lemma 1.3 in \cite{DKP}, it is 
		showed that
		\begin{equation}\label{eq:DKP2}
			\begin{aligned}
				&\dfrac{|u(x)-u(y)|^{p-2}(u(x)-u(y))}{|x-y|^{n+sp}}
				\left( v(x)-v(y)\right)\le\\
				&\le
				-\dfrac{1}{C}\dfrac{1}{|x-y|^{n+sp}}
				\left| \log\left(\dfrac{u(x)+\delta}{u(y)+\delta}
				\right)\right|^p\phi(y)^p+C\dfrac{|\phi(x)-\phi(y)|^{p}}
				{|x-y|^{n+sp}}
			\end{aligned}
		\end{equation}
		for a constant $C\equiv C(p).$ Moreover, in the  case  
		$B_{R}\subset\Omega,$ it is showed that
		\[
	  		\hh(u,(u+\delta)^{1-p}\phi^p) \le \ Cr^{n-sp} 
	  		 -\iint_{B_r\times B_r} \dfrac{1}{|x-y|^{n+sp}} 
	  		\left| \log\left(\dfrac{u(x)+\delta}{u(y)+\delta}
	  		\right)\right|^p \, dxdy,
		\]
		where $C$ independent on $\delta.$ Then, by \eqref{eq:DKP1}
		and using that $0\le u^{p-1}(u+\delta)^{1-p}\phi^p\le1$ in 
		$B_{\nicefrac{3r}{2}}\cap\Omega=B_{\nicefrac{3r}{2}},$ the lemma 
		holds.
	
		\medskip
	
		We proceed now to consider the case 
		$B_{R}\subset\R^n\setminus
		\overline{\Omega}.$ Since 
		$B_{\nicefrac{3r}2}\cap\Omega=\emptyset,$ by \eqref{eq:DKP1} 
		and \eqref{eq:DKP2},
		\[
			 \begin{aligned}
				 \iint_{B_r\times\Omega} \dfrac{1}{|x-y|^{n+sp}} \left| 
				 \log \left(\dfrac{u(x)+\delta}{u(y)+\delta}\right)
				 \right|^p \, dxdy &\le 
				 C \iint_{B_{\nicefrac{3r}{2}}\times \Omega} 
				 \dfrac{|\phi(x)|^{p}}
				 {|x-y|^{n+sp}}dx dy\\
				 &\le C \dfrac{r^{n}}{\mbox{dist}(B_R,\Omega)^{sp}}
				 \end{aligned}
		\]	
		for $C=C(n,s,p)$
	\end{proof}
	
	Proceeding as in the proof of Theorem A.1 in \cite{BF} 
	and using the previous lemma, we get the following minimum principle.

	\begin{thm}[Minimum Principle]\label{thm:mprinciple}
		Let $s,\ve\in (0,1)$ and $p\in(1,\infty).$ If $u$ is a weak 
		super-solution of \eqref{eq:a1} such that $u\ge0$ in $\R^n$ 
		and $u\not\equiv0$ in all connected 
		components of $\R^n\setminus\partial\Omega$, 
		then $u>0$ a.e in $\Omega$. 
	\end{thm}
	\begin{proof}
		We argue by contradiction and we assume that 
		$Z=\{x\colon u(x)=0\}$ has positive measure. Since 
		$u\not\equiv0$ in all connected 
		components of $\R^n\setminus\Omega,$ there are a ball
		$B_R(x_0)\subset\R^n\setminus\partial\Omega$ and 
		$r\in(0,2R)$ such that
		$|B_r(x_0)\cap Z|>0$ and $u\not\equiv0$ in $B_r(x_0).$ 
		
		For any $\delta>0$ and $x\in\R^n,$ we define 
		\[
			F_{\delta}(x)\coloneqq\log\left(1+\dfrac{u(x)}{\delta}
			\right).
		\] 
		
		Observe that, if $y\in B_r(x_0)\cap Z $ then
		\[
			|F_\delta(x)|^p=|F_\delta(x)-F_\delta(y)|^p
			\le \dfrac{(2r)^{n+sp}}{|x-y|^{n+sp}} \left|
			\log\left(\dfrac{u(x)+\delta}{u(y)+\delta}\right)\right|^p
			\quad\forall x\in\R^n.
		\]
		Then
		\[
			|F_{\delta}(x)|^p\le\dfrac{(2r)^{n+sp}}{|Z\cap B_r(x_0)|}
			\int_{B_r(x_0)}\dfrac{1}{|x-y|^{n+sp}} \left|
			\log\left(\dfrac{u(x)+\delta}{u(y)+\delta}\right)\right|^p
			dy\quad\forall x\in\R^n.
		\]
		Therefore
		\[
			\int_{A}
			|F_{\delta}(x)|^p dx\le\dfrac{(2r)^{n+sp}}{|Z\cap B_r(x_0)|}
			\iint_{B_r(x_0)\times A}\dfrac{1}{|x-y|^{n+sp}} \left|
			\log\left(\dfrac{u(x)+\delta}{u(y)+\delta}\right)\right|^p
			dxdy
		\]
		where 
		\[
			A=
			\begin{cases}
			B_r &\mbox{if }B_{R}\subset \Omega,\\
			\Omega &\mbox{if }B_{R}\subset 
			\R^n\setminus\overline{\Omega}.\\
			\end{cases}
		\]
		By, Lemma \ref{lema:DKP}, there is a constant $C$ independent
		on $\delta$ such that 
		\[
			\int_{A}
				|F_{\delta}(x)|^p dx\le
				C\dfrac{r^{2n}(1+r^{sp})}{|Z\cap B_r(x_0)|}.
		\]
		Taking $\delta \to 0$  in the above inequality, we obtain
		\[
			u\equiv0 \mbox{ in } A
		\]
		which is a contradiction since $u\not\equiv0$ in all connected 
		components of $\R^n\setminus\partial\Omega.$ 
		Thus $u>0$ in $\R^n.$
	\end{proof}
	
\section{The eigenvalue problem}\label{primer.atuvalor}
	In this section, we prove that 
	$\lambda_{1,\ve}(s,p)$ is the first non--zero eigenvalue of 
	\eqref{ecu}; that there is a sequence of eigenvalues; and that
	the eigenfunctions are bounded. Additionally, we show that
	$\lambda_{1,\ve}(s,p)$ is simple and isolated.
	Variational methods for non-local operators of elliptic type.
	For more datails about the construction of the eigenvalues in 
	nonlocal settings, see, for instance, \cite[Appendix A]{SV} and 
	\cite{LL}. 
	 
	\begin{thm}
		$\lam_{1,\ve}(s,p)$ is the first non-zero eigenvalue of \eqref{ecu}.
	\end{thm}
	\begin{proof}
		Take a minimizing sequence  $\{u_{k}\}_{k\in\N} 
		\subset\W(\Omega) $ of 
		$\lam_{1,\ve}(s,p)$ and normalize it according to
		$\|u_{k}\|_{L^p( \Omega_\ve)}=\ve.$ 
		Then, there is a constant $C$ such that
		$$
			\|u_{k}\|_{s,p} \leq C.
		$$
		Thus, by  Theorem \ref{teo:inclucomp}, up to a subsequence,
		\begin{align}\label{conv.1}
			\begin{split}
				&u_{k}\cd u \quad \mbox{weakly in } \W(\Omega),\\
				&u_{k}\to u \quad \mbox{strongly in } L^p(\Omega).
			\end{split}	
		\end{align}
		In particular, $u_{k}\to u$ strongly in $L^p(\Omega_{\ve})$ and 
		therefore 
		$\|u\|_{L^p(\Omega_\ve)}=\ve.$

		Since \eqref{conv.1} holds, 
		\begin{align*}
			\C(1-s)[u]_{s,p}^p + 
				\|u\|^p_{L^p(\Omega)} &
			\leq \liminf_{k\to \infty} 
			\C(1-s)[u_{k}]_{s,p}^p + \|u_{k}\|_{L^p(\Omega)}^2\\
			&=\lim_{k\to \infty}\C (1-s)
			[u_{k}]_{s,p}^p + \|u_{k}\|_{L^p(\Omega)}^p\\
			& = \lam_{1,\ve}(s,p).
		\end{align*}
		Then, by \eqref{lam1}, we have that
		\[
			\C(1-s)
			[u]_{s,p}^p + \|u\|_{L^p(\Omega)}^p=\lam_{1,\ve}(s,p).
		\]

		The fact that a minimizer verifies \eqref{debil} is standard but 
		we include a short proof here for the sake of completeness. 
		Let $u$ be a nontrivial minimizer of \eqref{lam1}. 
		Then, using Lagrange's multipliers, we get the existence of a 
		value $\lambda \in {\mathbb{R}}$ such that
		\begin{equation} \label{fff}
			\C(1-s) \hh(u,v)   + \int_{\Omega} |u_p|^{p-2}uv
			\, dx = \dfrac{\lam}{\ve} 
			\int_{ \Omega_\ve} |u|^{p-2}uv\, dx.
		\end{equation}
		for all $v\in \W(\Omega)$ with $ \|v\|_{L^p(\Omega_\ve)}=\ve$. 
		Therefore \eqref{fff} also holds for all $v\in \W(\Omega).$
		Finally, taking $v=u$ we get that $\lambda = \lam_{1,\ve}(s,p)$.
	\end{proof}
	
	Using a topological tool (the genus), we can construct
	an unbounded sequence of eigenvalues.
		
	\begin{thm} \label{teo2}
		There is a sequence of eigenvalues $\lambda_{k,\epsilon} (s,p)$ 
		such that $\lambda_{k,\epsilon} (s,p) \to\infty$ as $k\to\infty$.
	\end{thm}
		
	\begin{proof} 
		We follow ideas from \cite{GAP} and hence we omit the details.
		Let us consider 
		\[
			M_\alpha = \{u \in \W(\Omega) \colon  \|u\|_{s,p}= p 
			\alpha \}
		\] 
		and
		\[ 
			\varphi (u) = \frac{1}{p}
			\int_{\Omega_\epsilon} |u(x)|^p\, dx.
		\]
		We are looking for critical points
		of $\varphi$ restricted to the manifold $M_\alpha$ using a minimax
		technique.
		We consider the class
		\[
			\Sigma = \{A\subset \W(\Omega) \setminus\{0\}
			\colon A \mbox{ is closed, } A=-A\}.
		\]
		Over this class we define the genus, 
		$\gamma\colon\Sigma\to {\mathbb{N}}\cup\{\infty\}$, as 
		\[
			\gamma(A) = \min\{k\in {\mathbb{N}}\colon
			\mbox{there exists } \phi\in C(A,{{\mathbb{R}}}^k-\{0\}),
			\ \phi(x)=-\phi(-x)\}.
		\]
		Now, we let $C_k = \{ C \subset M_\alpha \colon C 
		\mbox{ is compact, symmetric and } \gamma ( C) \le k \} $ 
		and let
		\begin{equation}
			\label{betak} \beta_k 
			= \sup_{C \in C_k} \min_{u \in C} \varphi(u). 
		\end{equation}
		Then $\beta_k >0$ and there exists $u_k \in
		M_\alpha$ such that $\varphi (u_k) = \beta_k$ and $ u_k$ is a weak
		eigenfunction with $\lambda_k = \alpha / \beta_k $.
	\end{proof}
	
	Our next aim is to prove that the eigenfunctions are bounded.
	We follow ideas from \cite{FP}.
	\begin{lema}\label{lema:cotainfty}
		Let $s,\ve\in (0,1),$ $p\in(1,\infty),$ 
		and $\lambda$ be an eigenvalue
		of \eqref{ecu}. If $u$ is an eigenfunction associated to 
		$\lambda$ then $u\in L^{\infty}(\Omega).$
	\end{lema}
	\begin{proof}
		If $ps>n,$ by  Theorem \ref{teo:inclucomp}, 
		then the assertion holds.
		Then let us suppose that $sp \le n$.
		We will show that if $\|u_+\|_{L^p(\Omega)} \le \delta$ 
		then $u_+$ is bounded, where $\delta > 0$ must be determined.
				
		For $k\in\mathbb{N}_0$ we define the function $u_k$ by
		\[
			u_k\coloneqq(u(x)-1+2^{-k})_+.
		\]
		Observe that, $u_0= u_+$ and for any $k\in\mathbb{N}_0$ 
		we have that $u_k\in\W(\Omega)$, 
		\begin{equation}\label{eq:lci1}
			\begin{aligned}
				&u_{k+1}\le u_k \text{ a.e. } \mathbb{R}^n,\\
				&u <(2^{k+1}-1)u_k\text{ in } \{u_{k+1}>0\},\\
				&\{u_{k+1}>0\}\subset\{u_k>2^{-(k+1)}\}.
			\end{aligned}
		\end{equation}
		Now, since
		\[
			|v_+(x)-v_+(y)|^p\le 
			|v(x)-v(y)|^{p-2}(v(x)-v(y))(v_+(x)-v_+(y))
			\quad\forall x,y\in\mathbb{R}^n,
		\]
		for any function $v \colon \R^n \to \R$, by taking $v=u-1+2^{-k}$ we have that
		$$
		\begin{array}{l}
		\displaystyle
			\C(1-s)[u_{k+1}]_{s,p}^p + \|u_{k+1}\|_{L^p(\Omega)}^p
			\\ [10pt]
			\qquad \displaystyle
			\le\C(1-s) \hh(u,u_{k+1})+ \int_{\Omega} |u|^{p-2}u u_{k+1}\, dx\\[10pt]
			\qquad \displaystyle = \dfrac{\lam}{\ve} 
			\int_{ \Omega_\ve} |u|^{p-2}uu_{k+1}\, dx,
		\end{array}
		$$
		for all $k\in\mathbb{N}_0.$
		Then, by \eqref{eq:lci1}, we have that
		\begin{equation}\label{eq:lci2}
			\begin{aligned}
				\C(1-s)[u_{k+1}]_{s,p}^p 
				+ \|u_{k+1}\|_{L^p(\Omega)}^p
				&\le  \dfrac{\lam}{\ve}\int_{\Omega_\ve} u^{p-1} u_{k+1} 
				\, dx\\ 
				&\le  \dfrac{\lam}{\ve}
				(2^{k+1}-1)^{p-1} \|u_k\|_{L^p(\Omega)}^p 
			\end{aligned}
		\end{equation}
		for all $k\in\mathbb{N}_0$. 
			
		On the other hand, in the case $sp<n,$ using H\"older's 
		inequality, fractional Sobolev embeddings and Chebyshev's 
		inequality, for any $k\in\mathbb{N}_0$ we have that 
		\begin{equation}\label{eq:lci3}
			\begin{aligned}
				\|u_{k+1}\|_{L^p(\Omega)}^p &\le 
				\|u_{k+1}\|_{L^{p_s^*}(\Omega)}^p 
				|\{u_{k+1}>0\}|^{\nicefrac{sp}{n}} \\
				&\le C\|u_{k+1}\|_{s,p}^p |\{u_{k+1}>0\}
				|^{\nicefrac{sp}{n}} \\
				&\le C\|u_{k+1}\|_{s,p}^p
				|\{u_{k}>2^{-(k+1)}\}|^{\nicefrac{sp}{n}}\\
				&\le  C\|u_{k+1}\|_{s,p}^p \left( 2^{(k+1)p} 
				\|u_k\|_{L^p(\Omega)}^p \right)^{\nicefrac{sp}n}.
			\end{aligned}
		\end{equation}
		Similarly, in the case $sp=n,$ taking $r>p$ and proceeding as in 
		the previous case $sp<n$ (with $r$ in place of $p_s^*$), 
		we have that \eqref{eq:lci3} holds with 
		$1-\nicefrac{p}r>0$ in place of $\nicefrac{sp}n$.
				
		Then, by \eqref{eq:lci2} and \eqref{eq:lci3},  
		there exist a constant $C>1$ and 
		$\alpha>0$ both independent on $k$ such that
		\[
			\|u_{k+1}\|_{L^p(\Omega)}^p\le 
			C^k(\|u_k\|_{L^p(\Omega)}^p)^{1+
				\alpha}.
		\]
		Therefore, if $\|u_+\|_{L^p(\Omega)}^p=\|u_0\|_{L^p(\Omega)}^p
		\le C^{-\nicefrac{1}{\alpha^2}}=\delta^p$ then 
		\[
			\lim_{k\to+\infty}\|u_{k}\|_{L^p(\Omega)}=0.
		\]	
		On the other hand, as $u_k\to(u-1)_+$ a.e in $\mathbb{R}^n,$ 
		we obtain $(u-1)_+\equiv0$ in $\mathbb{R}^n.$ Therefore $u_+$ is 
		bounded. 
				
		Finally, taking $-u$ in place of $u$ we have  
		that $u_-$ is bounded if $\|u_-\|_{L^p(\Omega)} < \delta$. 
		Therefore $u$ is bounded.     
	\end{proof}
					
	Now, using Theorem \ref{thm:mprinciple}, we show that
	a non-negative eigenfunction is positive.
			
	\begin{lema}\label{lema:positivo}
			Let $s,\ve\in (0,1),$ $p\in(1,\infty).$ 
			and $\lambda$ be an eigenvalue
			of \eqref{ecu}. If $u$ is a non-negative 
			eigenfunction associated to 
			$\lambda$ then $u>0$ in $\R^n.$
	\end{lema}	
	\begin{proof}
		By Theorem \ref{thm:mprinciple}, we only need to show that
		$u\not\equiv0$ in all connected components of 
		$\R^n\setminus\partial\Omega.$ Suppose, by contradiction, that
		there is $Z$ a connected components of 
		$\R^n\setminus\partial\Omega$ such that $u\equiv0$ in $Z.$
		Taking $\phi\in C_0^{\infty}(Z)$ as a test 
		function in \eqref{debil}, we get
		\[
			\hh(u,\phi)=0.
		\]
		Therefore
		\[
			\int_{\Omega}(u(x))^{p-1}
			\int_{Z}\dfrac{\phi(y)}{|x-y|^{n+sp}}dydx=0
			\qquad\forall \phi\in C_0^{\infty}(Z).
		\]
		Then $u=0$ in $\Omega.$ Thus, since $u$ s a non-negative 
		eigenfunction associated to $\lambda,$ we obtain that
		\[
			[u]_{s,p}=\hh(u,u)=
			\dfrac{1}{\C(1-s)}
			\left(\dfrac{\lambda}{\ve}
			\int_{\Omega_{\ve}}|u|^p\, dx- \int_{\Omega}|u|^p\, dx
			\right)=0.
		\]
		Hence $u\equiv0$ in $\R^n$ which is a contradiction since
		$u\not\equiv0$ in $\R^n.$
	\end{proof}

	Note that, if $u$ is an eigenfunction associated to 
	$\lambda_{1,\ve}(s,p)$ then  
	\[
		u_{+}(x)=\max\{u(x),0\}\not\equiv 0 
		\mbox{ or } 
		u_{-}(x)=\max\{-u(x),0\}\not\equiv0
	\] 
	in $\Omega_{\ve}$. If $u_{+}(x)\not\equiv0$ in 
	$\Omega_{\epsilon},$ then
	\begin{align*}
		\C(1-s)[u_+]_{s,p}^p + 
			\|u_+\|^p_{L^p(\Omega)} &
		\leq  
		\C(1-s)\hh(u,u_+) 
		+ \int_{\Omega}|u|^{p-2}uu_{+} dx\\
		& = \dfrac{\lam_{1,\ve}(s,p)}{\ve}\int_{\Omega_\ve}|u|^{p-2}uu_{+} 
		dx\\	
		& = \dfrac{\lam_{1,\ve}(s,p)}{\ve}\|u_+\|_{L^p(\Omega_\ve)}^p,
	\end{align*}
	that is, $u_+$ is a minimizer of \eqref{lam1}. Therefore $u_+$ is a
	non-negative eigenfunction associated to $\lambda_{1,\ve}(s,p).$ Then,
	by Lemma \ref{lema:positivo}, $u_{+}>0$ in $\Omega.$

	In the same manner we can see that if $u_{-}(x)\not\equiv0$
	in $\Omega_{\epsilon},$ then $u_{-}>0$ in $\Omega.$ 
	Thus the next theorem is proved.

	\begin{thm}\label{thm:signoconstante}
		Any eigenfunction associated to $\lambda_{1,\ve}(s,p)$ 
		has constant sign. 
	\end{thm}

	A key ingredient in the next sections is the simplicity of the 
	first eigenvalue $\lambda_{1,\ve}(s,p)$. 
	In order to prove this result we need the 
	following Picone-type identity (see Lemma~6.2  in \cite{Amghibech}).

	\begin{lema}\label{lem:laux1}  
		Let $p\in(1,\infty)$. For $u,v\colon\mathbb{R}^n\to\mathbb{R}$ such
		 that $u \ge 0$ and $v>0$,  we have
		\[
			L(u,v)\ge0 \quad \mbox{in } \R^n \times \R^n,
		\]
		where
		\[
	  		L(u,v)(x,y) = |u(x)-u(y)|^p-|v(x)-v(y)|^{p-2}(v(x) - v(y))
	  		\left(\dfrac{u^p(x)}{v^{p-1}(x)} 
	  		- \dfrac{u^p(y)}{v^{p-1}(y)}\right). 
		\]
		The equality holds if and only if $u=kv$ a.e. in 
		$\mathbb{R}^n$ for some constant $k$.
	\end{lema}
	
	\begin{thm}\label{thm:autoval1}
  			Let $\Omega \subset \R^n$ be a bounded open connected 
  			set with Lipschitz boundary. 
  			Assume that $u$ is a positive eigenfunction corresponding to 
  			$\lambda_{1,\ve}(s,p).$ 
  			Then if $\lambda > 0$ is such that there exists a 
  			non-negative eigenfunction $v$ of \eqref{ecu} with eigenvalue 
  			$\lambda,$ then $\lambda = \lambda_{1,\ve}(s,p)$ and 
  			there exists $c \in \R$ such that $v = c u$ a.e. in 
  			$\R^n$. 
	\end{thm}
	\begin{proof} 
 		 Since $\lambda_{1,\ve}(s,p)$ is the first eigenvalue, we have that 
 		 $\lambda_{1,\ve}(s,p) \le \lambda$. 
 		 On the other hand, by Lemma \ref{lema:positivo}, $v>0$ in 
 		 $\R^n.$
 		 
 		 For $k \in \N$ take
 		 $v_k\coloneqq v + \nicefrac1{k}$. 
 		 We begin by proving that $w_{k} \coloneqq u^{p} / v_k^{p-1}
 		 \in\W(\Omega)$. 
        First observe that $w_{k}\in L^{p}(\Omega),$ due to 
        $u\in L^{\infty}(\Omega)$, see Lemma \ref{lema:cotainfty}. 
        Now, for all $(x,y) \in \R^n \times \R^n$ we have
  		\begin{align*}
    		|w_{k}(x)-w_{k}(y)| =& 
    		\left| \dfrac{u^p(x) - u^p(y)}{v_k^{p-1}(x)} - 
    		\dfrac{u^p(y) 
    		\left( v_k^{p-1}(x) - v_k^{p-1}(y) 
    		\right)}{v_k^{p-1}(x) v_k^{p-1}(y)} \right| \\
			\le &\ k^{p-1} \left|u^p(x) - u^p(y) \right| + 
			\|u\|_{\infty}^p 
			\dfrac{\left|v_k^{p-1}(x) - v_k^{p-1}(y) \right|}
			{v_k^{p-1}(x)v_k^{p-1}(y)} \\
			\le&\ p k^{p-1} (u^{p-1}(x) + u^{p-1}(y))|u(x) - u(y)| \\
			& + (p-1)\|u\|_{\infty}^p 
			\dfrac{v_k^{p-2}(x) + v_k^{p-2}(y)}{v_k^{p-1}(x) v_k^{p-1}(y)}
			|v_k(x)-v_k(y)|\\
			\le &\ 2pk^{p-1}\|u\|_\infty^p |u(x)-u(y)| \\
			& + (p-1) \|u\|_{\infty}^p
			\left(\dfrac1{v_k(x) v_k^{p-1}(y)} 
			+ \dfrac1{v_k^{p-1}(x) v_k(y)} \right)|v(x)-v(y)| \\
	 		\le &\ C(k,\|u\|_\infty,p)
			\left( |u(x)-u(y)|+|v(x)-v(y)| \right).
	  	\end{align*} 
  		As $u,v \in \W(\Omega)$, we deduce that 
  		$w_k \in W(\Omega)$ for all $k \in \N$.

  		Recall that $u,v\in\W(\Omega)$ are two eigenfunctions of problem  		
  		\eqref{ecu}  with eigenvalue $\lambda_1(s,p)$ 
  		and $\lambda$ respectively. Then, by using the previous lemma, 
  		we deduce that 
  		\begin{align*}
    		0 \le& 
    		\C(1-s)\iint_{\R^{2n}\setminus{(\Omega^c)^2}} 
    		\dfrac{L(u,v_k)(x,y)}{|x-y|^{n+sp}} \, d x d y \\
    		\le &\ \C(1-s)\iint_{\R^{2n}\setminus{(\Omega^c)^2}} 
    		\dfrac{|u(x)-u(y)|^p}{|x-y|^{n+sp}} \, d x d y \\ 
    		& - \C(1-s) 
    		\iint_{\R^{2n}\setminus{(\Omega^c)^2}}
    		\dfrac{|v(x)-v(y)|^{p-2}(v(x)-v(y))}{|x-y|^{n+sp}}	\\
    		& \qquad \qquad \qquad \qquad \times \left(\dfrac{u^p(x)}{v_k^{p-1}(x)} 
    		- \dfrac{u^p(y)}{v_k^{p-1}(y)} \right) \, dxdy \\
			\le &\ \C(1-s)\iint_{\R^{2n}\setminus{(\Omega^c)^2}} 
    		\dfrac{|u(x)-u(y)|^p}{|x-y|^{n+sp}} \, d x d y 
    		\\ & - \dfrac{\lambda}{\ve} \int_{\Omega_\ve} 
			v^{p-1} \dfrac{u^p}{v_k^{p-1}}\, d x+\int_{\Omega}v^{p-1}
			\dfrac{u^p}{v_k^{p-1}}\, d x\\
			\le &\dfrac{\lambda_{1,\ve}(s,p)}{\ve} 
			\int_{\Omega_\ve} u^p \, d x - \int_{\Omega}|u|^p \, dx  
			- \dfrac{\lambda}{\ve} \int_{\Omega_\ve} 
			v^{p-1} \dfrac{u^p}{v_k^{p-1}}\, d x+\int_{\Omega}v^{p-1}
			\dfrac{u^p}{v_k^{p-1}}\, d x.
		\end{align*}
		Taking $k \to \infty$ and using Fatou's lemma and 
		the dominated convergence theorem, we infer that
		\[
      		\iint_{\R^{2n}\setminus{(\Omega^c)^2}} 
     		 \dfrac{L(u,v)(x,y)}{|x-y|^{n+sp}} \, dx dy = 0
		\]
		(recall that $\lambda_{1.\ve}(s,p)\le \lambda$).
		Therefore, by the previous lemma, $L(u,v)(x,y)=0$ a.e. in
		$\R^{2n}\setminus{(\Omega^c)^2}$
		and $u=cv$ for some constant $c>0$.
\end{proof}
    
We will need the following lemma.
	
	\begin{lema}\label{lema:auxiaut}
				Let $\ve>0$.
				If $u$ is an eigenfunction associates to $\lambda
				>\lambda_{1,\ve}(s,p)$ there exist 
				$C>0$ and $\alpha>0$ independent on $\lambda,$ 
				$u$ and $\ve$ such that
				\[
					\left(\dfrac{C\ve}{\lambda}
					\right)^\alpha\le|\Omega^{\pm}|.
				\] 
				 Here $\Omega^+=\{x\in\Omega\colon u(x)>0\},$ 
				 and $\Omega^-=\{x\in\Omega\colon u(x)<0\}.$ 
	\end{lema}
	\begin{proof} 
		Let $u^+(x)=\max\{0,u(x)\}.$ 
		Since $u$ is an eigenfunction associates to 
		$\lambda>\lambda_{1,\ve}(s,p),$ $u$ changes sign then $u^+
		\not\equiv0.$
		In addition,
		\begin{equation}\label{eq:aumedes1}
			\begin{aligned}
				\min\left\{\C(1-s),1\right\}
				\|u^{+}\|_{s,p}^p&\le
				\C(1-s)[u^+]_{s,p}^p + 
					\|u^+\|^p_{L^p(\Omega)}\\ 
					&\leq  
				\C(1-s)\hh(u,u^+) 
				+ \int_{\Omega}|u|^{p-2}uu^{+} dx\\
				& = \dfrac{\lam}{\ve}\int_{\Omega_\ve}|u|^{p-2}uu^{+} 
				dx\\	
				& = \dfrac{\lam}{\ve}\|u^+\|_{L^p(\Omega_\ve)}^p.
			\end{aligned}
		\end{equation}
	
		On the other hand, by Sobolev embedding theorem, there exists 
	    a constant $C$ independent on $\lambda,$ $u$ and $\ve$ such
		that
		\[
			\|u^+\|_{L^q(\Omega)}
			\le C\|u^+\|_{s,p}
	    \]
	    where $1<q<p_{s}^\star.$ 
		Then, by \eqref{eq:aumedes1} and H\"older's inequality, 
		there exists a constant $C$ independent on $\lambda,$ $u$ and $\ve$ such
		that
		\[
			\|u^+\|^p_{L^{q}(\Omega)}
			\le C\dfrac{\lambda}\ve
			\|u^{+}\|_{L^{q}(\Omega)}^{p}
			   |\Omega^+|^{\frac{q-p}{q}}
		   \quad\forall p<q<p_{s}^\star.
	    \]
	    Fix any $p<q<p_{s}^\star$ and take $\alpha=\dfrac{q}{q-p}$
	    \[
			\left(\dfrac{\ve}{C\lambda}\right)^{\alpha}\le|\Omega^{+}|.
    	\]

		In order to prove the second inequality, it will suffice 
		to proceed as above, using the function $u^-(x)=\max\{0,-u(x)\}$ 
		instead of
		$u^+.$
	\end{proof}

	\begin{thm}\label{thm:isolated}
		For each fixed value $\ve>0$, 	$\lam_{1,\ve}(s,p)$ is isolated.
	\end{thm}
	
	\begin{proof}
		From its definition, we have
		that $\lambda_{1,\ve}(s,p)$ is left--isolated.
	
		To prove that $\lambda_{1,\ve}(s,p)$ is right--isolated, 
		we argue by contradiction. We assume that there exists a  sequence 
		of  eigenvalues $\{\lambda_k\}_{k\in\N}$ such that  
		$\lambda_k>\lambda_{1,\ve}(s,p)$ and $\lambda_k\searrow 
		\lambda_{1,\ve}(s,p)$ as $k\to +\infty.$  
		Let $u_k$ be an eigenfunction associated to $\lambda_k,$ 
		we can assume that 
		\[
			\frac1\ve\int_{\Omega_{\ve}}|u_k(x)|^p\,dx=1.
		\] 
		Then $\{u_k\}_{k\in\N}$
		is bounded in $\W(\Omega)$ and therefore
		we can extract a subsequence 
		(that we still denoted by $\{u_k\}_{k\in\N}$) 
		such that 
		\begin{align*}
			u_k\rightharpoonup u&\quad \mbox{ weakly in } \W(\Omega),\\
			u_k\to u &\quad \mbox{ strongly in } L^p(\Omega).
		\end{align*}
		Then 
		\[
			\dfrac1\ve\int_{\Omega_{\ve}}|u(x)|^p\,dx=1
		\]
		and
		\begin{align*}
			\C(1-s)[u]_{s,p}^p+\|u\|_{L^p(\Omega)}^p&
			\le\C(1-s)
			\liminf_{k\to+\infty}[u_k]_{s,p}^p+\|u\|_{L^p(\Omega)}^{p}\\
					&=\lim_{k\to+\infty}\lambda_k=\lambda_{1,\ve}(s,p).
		\end{align*}
		Hence, $u$ is an eigenfunction associates to $\lambda_{1,\ve}(s,p).$ 
		By Theorem \ref{thm:signoconstante}, we can assume that $u>0.$
	
		On the other hand, by the Egorov's theorem, for any $\delta>0$
		there exists a subset $A_\delta$ of $\Omega$ such that 
		$|A_\delta|<\delta$ and $u_k\to u>0$ uniformly in 
		$\Omega\setminus A_{\delta}.$ This contradicts the fact that,
		by Lemma \ref{lema:auxiaut},
		\[
						\left(\dfrac{C\ve}{\lambda_k}
						\right)^\alpha
						\le|\{x\in\Omega\colon u_k(x)<0\}|.
		\]
		This proves the theorem.		
	\end{proof}

\section{The limit of $\lam_{1,1-s}(s,p)$ as $s\to 1^-$.} \label{sect-limite}
	
	Throughout this section, we assume that $\Omega$ is a 
	smooth bounded domain and take $\ve=1-s.$
	
	Here we analyze the behavior of $\lam_{1,1-s}(s,p)$ as $s\to 1^-$. 
	For simplicity, we omit the subscript $1-s$ and we just write
	$\lam_1(s,p)$.

	\medskip
	
	First we show that 
	\[
		\limsup_{s\to1^-}\lambda_1(s,p)\le \lambda_1(p).
	\]
	For this purpose, we state some convergence results. 
	We start with the following lemma.
	
	\begin{lema}  \label{int1}
		Let $\Omega$ be a domain in $\R^n$ with Lipschitz boundary
		and $p\in(1,\infty)$.
		If $u\in W^{1,p}(\Omega)$ then
		\begin{equation*}
		\lim_{\ve \to 0^+}\frac{1}{\ve} \int_{\Omega_\ve} |u|^p\, dx =\int_{\partial \Omega} |u|^p \, dS.
		\end{equation*}
	\end{lema}	
	
	In order to deal with the integrals on $\Omega_{\ve}$ we will state 
	the following lemma, which is an immediate consequence of the Coarea 
	formula. See \cite[Section 3.4.4]{EvGa},  for details. 
	
	\begin{lema} \label{lema.coarea}
		Given $g:\R^n \to \R$ an integrable function, and $f:\R^n\to\R$  
		a Lipschitz function such that $\mbox{essinf}\,|Df|>0$. Then it 
		follows that
		\begin{equation} \label{coarea}
		\int_{\{0<f<t\}} g\, dx= \int_0^t\left( \int_{\{f=r\}} 
		\frac{g}{|Df|}dS\right)dr.
		\end{equation}		
	\end{lema}
	
	Now we are ready to proceed with the proof of Lemma \ref{int1}.
	
	\begin{proof}[Proof of Lemma \ref{int1}]
		We consider the $(n-1)$-dimensional hyper-surface in $\R^n$ given 
		by $\omega_r=\{x\in\R^n \colon \, d(x,\Omega^c)=r\}$, where 
		$d(x,\Omega)=\inf_{y\in\Omega} |x-y|$. Observe that
		$\Omega_\ve=\{x\in \R^n \colon \, x\in \omega_r \mbox{ for } r\in 
		[0,\ve]\}$ and $\omega_0=\partial \Omega$.
		By applying Lemma \ref{lema.coarea} with $g=|u|^p$ and $f(x)=d(x,\Omega^c)$ we get
		$$
		\int_{\Omega_\ve} |u|^p\, dx = \int_0^\ve \left(\int_{\omega_r} |u|^p \, dS\right)dr
		$$
		since $|Df|=1.$
		The Mean value theorem for integrals asserts that there exists $r_0\in [0,\ve]$ such that
		$$
		\int_0^\ve \left(\int_{\omega_r} |u|^p \, dS\right)dr= \ve  \int_{\omega_{r_0}} |u|^p \, dS.
		$$
		Since $r_0$ tends to $0$ as $\ve\to 0^+$, we get  that $\omega_{r_0}$ tends to  $\partial \Omega$ as $\ve \to 0^+$ and the result follows using that the trace operator for a fixed function in $W^{1,p} (\Omega)$ 
		depends continuously on the hyper-surface (see \cite{ARBR}) and hence
		$$
		\int_{\omega_{r_0}} |u|^p \, dS \to \int_{\partial \Omega} |u|^p \, dS, 
		$$
		as $r_0 \to 0+$.
	\end{proof}
	
	If $\Omega$ is a smooth bounded domain in $\R^n$ then, by Theorem 7.25 in \cite{GT}, for any open ball $B_R\supset\supset \Omega$
	there is a bounded linear extension operator $E$ from 
	$W^{1,p}(\Omega)$ into $W^{1,p}_0(B_R)$ such that 
	$Eu=u$ in $\Omega.$ Our next goal is to prove that
	\begin{equation}\label{eq:limiteimportante}
		\mathcal{K}_{n,p}(1-s)[Eu]_{s,p}\to\|\nabla u\|_{\lp}^p
	\end{equation}
	as $s\to1^-.$  To this end, we need the following result.
	For the proof we refer to \cite[Corollary 2]{BBM}.
	
	\begin{thm}
		\label{teo:bbm1}	
		Let	$\Omega$ be a smooth bounded domain and $p\in(1,\infty).$
		Assume $u\in L^p(\Omega),$ then
		\[
		\lim_{s\to 1^-}\C(1-s)|u|_{ W^{s,p}(\Omega)}^p=|u|_{ W^{1,p}(\Omega)}^p 
		\]
		with
		\[
		|u|_{ W^{1,p}(\Omega)}^p =\begin{cases}
		\|\nabla u\|^p_{\lp} 
		&\text{ if } u\in W^{1,p}(\Omega),\\
		\infty &\text{ otherwise}.
		\end{cases}	
		\]	
	\end{thm}
	
	We now show \ref{eq:limiteimportante}, 
	which will be key  in the proof of next results. 
	
	\begin{lema}  \label{int2}
		If $u\in W^{1,p}(\Omega),$ then
		$$
		\lim_{s\to 1^-}\C(1-s)[Eu]_{s,p}^p =
		\|\nabla u\|_{\lp}^p. 
		$$
	\end{lema}
		
	\begin{proof}
		Observe that 
		\begin{align*}
		[Eu]_{s,p}^p=|u|_{W^{s,p}(\Omega)}^p
		&+2\iint_{\Omega\times(\Omega^c\cap B_R)}
		\dfrac{|Eu(x)-Eu(y)|^p}{|x-y|^{n+sp}}\, dy dx\\
		&+2\iint_{\Omega\times B_R^c}
		\dfrac{|Eu(x)-Eu(y)|^p}{|x-y|^{n+sp}}\, dy dx.
		\end{align*}
		Then, by Theorem \ref{teo:bbm1}, we need to show that
		\begin{align*}
		&(1-s)
		\iint_{\Omega\times(\Omega^c\cap B_R)}
		\dfrac{|Eu(x)-Eu(y)|^p}{|x-y|^{n+sp}}\, dy dx\to0,\\
		&(1-s)\iint_{\Omega\times B_R^c}
		\dfrac{|Eu(x)-Eu(y)|^p}{|x-y|^{n+sp}}\, dy dx
		\to 0,
		\end{align*}
		as $s\to1^{-}.$
		
		By Theorem \ref{teo:bbm1}, what we have  that 
		\begin{align*}
		&\mathcal{K}_{n,p}(1-s)|Eu|_{W^{s,p}(B_R)}^p\to \|\nabla 
		Eu\|_{L^p(B_R)}^p,\\
		&\mathcal{K}_{n,p}(1-s)|Eu|_{W^{s,p}(\Omega)}^p\to \|\nabla Eu\|_{\lp}^p,\\
		&\mathcal{K}_{n,p}(1-s)|Eu|_{W^{s,p}(\Omega\cap
			B_R^c)}^p\to \|\nabla Eu\|_{L^p(\Omega^c\cap B_R)}^p,
		\end{align*}
		as $s\to1^{-}.$ Therefore 
		\begin{align*}
		(1-s)
		\iint_{\Omega\times(\Omega^c\cap B_R)}&
		\dfrac{|Eu(x)-Eu(y)|^p}{|x-y|^{n+sp}}\, dy dx\\
		&=\dfrac{(1-s)}2\left(
		|Eu|_{W^{s,p}(B_R)}^p-|Eu|_{W^{s,p}(\Omega)}^p
		-|Eu|_{W^{s,p}(\Omega\cap
			B_R^c)}^p\right)\\
		&\to 0 \quad\mbox{ as } s\to 1^-.
		\end{align*}
		
		On the other hand
		\[		
		(1-s)\iint_{\Omega\times B_R^c}
		\dfrac{|Eu(x)-Eu(y)|^p}{|x-y|^{n+sp}}\, dy dx
		\le C_n\dfrac{(1-s)}{sp}
		\dfrac{1}{d(\Omega,B_R^c)^{sp}}\to 0
		\]
		as $s\to 1^-.$
	\end{proof}	
	
	From Lemmas \ref{int1} and \ref{int2}, we get
	\begin{cor}  \label{int3} 
		Let	$\Omega$ be a smooth bounded domain and $p\in(1,\infty).$ 
		For a fixed $u\in W^{1,p}(\Omega)\setminus W^{1,p}_0(\Omega),$ 
		it holds
		$$
			\lim_{s\to 1^-}
			\dfrac{\C(1-s)[Eu]_{s,p}^p+\|Eu\|_{\lp}^p}
			{\frac{1}{1-s}\|Eu\|_{L^p(\Omega_{1-s})}^p} =
			\dfrac{\|\nabla u\|_{\lp}^p +\|u\|_{\lp}^p}
			{\|u\|_{L^{p}(\partial\Omega)}^p}. 
		$$
	\end{cor}
	
	From this result the following corollary is straightforward.
	
	\begin{cor}  \label{limsup}
		Let	$\Omega$ be a smooth bounded domain and $p\in(1,\infty).$ Then
			\[
				\limsup_{s\to1^-}\lambda_1(s,p)\le \lambda_1(p).
			\]	
	\end{cor}
	
	\medskip
	
	With this result in mind, to prove the last part of Theorem \ref{teo.1}, we need to show that
	\begin{equation}\label{eq:liminf}
		\lambda_1(p)\le	\liminf_{s\to1^-}\lambda_1(s,p).
	\end{equation}
	Before proving this, we need to state some auxiliary results.
	
	\medskip
	
	The next theorem is established in \cite[Corollary 7]{BBM}.
		
	\begin{thm}\label{BBM}
		Let	$\Omega$ be a smooth bounded domain, $p\in(1,\infty),$ 
		and $u_s \in W^{s,p}(\Omega).$ Assume that
		\[
			\|u_s\|_{L^p(\Omega)}\le C \quad \mbox{ and }
			\quad (1-s)|u_s|_{W^{s,p}(\Omega)}<C \quad
			\forall s>0.
		\]
		Then, up to a subsequence, $\{u_s\}$ converges in $\lp$
		(and, in fact, in $W^{s_0,p}(\Omega)$ for all $s_0\in(0,1)$)
		to some $u\in W^{1,p}(\Omega).$
	\end{thm}
		
	The proof of the following proposition can be found in 
		\cite[Proposition 3.10]{BPS}.
	
	\begin{prop}\label{prop:BPS}
		Let	$\Omega$ be a smooth bounded domain and $p\in(1,\infty).$
		Given $\{s_k\}\subset(0,1)$ an increasing sequence converging to
		1 and $\{u_k\}_{k\in\N}\subset\lp$ converging to $u$ in 
		$\lp,$ we have that
		\[
			\|\nabla u\|_{\lp}^p\le \lim_{k\to \infty}\mathcal{K}_{n,p}
			(1-s_k)|u_{s_k}|_{W^{s_k,p}(\Omega)}^p
		\]
	\end{prop}
	
	\begin{lema}  \label{lema.ricardo}
		Let	$\Omega$ be a smooth bounded domain, $p\in(1,\infty)$
		and $\{u_s\}_{s\in(0,1)}$ be such that $u_s\to u$ strongly in 
		$W^{t,p}(\Omega)$ for some $t\in(\nicefrac{1}{p},1).$ Then
		\[
			\dfrac{1}{1-s}\int_{\Omega_{1-s}} |u_s|^p dx\to 
			\int_{\partial\Omega} |u|^p dS
		\]
		as $s\to1^-.$
	\end{lema}
	\begin{proof}
		We start observing that, since $\partial\Omega\in C^2$ and 
		$t>\nicefrac1p,$ the 
		trace constant in the embedding $W^{t,p}(\Omega)\hookrightarrow 
		L^p(\partial \Omega_\varepsilon)$ for all $\varepsilon\in
		(0,\varepsilon_0)$ is bounded uniformly (independently of $\varepsilon).$ Then, 
		there is a constant $C$ independent on
		$s$ such that
		\[
			\|u_s-u\|_{L^p(\partial\Omega_{\varepsilon})}
			\le C	
			\|u_s-u\|_{W^{t,p}(\Omega)}.
		\]
		Therefore
		\[
			\dfrac{1}{1-s}\int_{\Omega_{1-s}}|u_s(x)|^p dx
			=\dfrac{1}{1-s}\int_0^{1-s}
			\left(\int_{\partial\Omega_{r}}
			|u_{s}|^p dS \right)dr\to\int_{\partial\Omega} |u|^p dS
		\]
		as $s\to1^{-}.$
	\end{proof}

	Now we are ready to prove \eqref{eq:liminf}.
	
	\begin{cor}  \label{liminf}
		Let	$\Omega$ be a smooth bounded domain and $p\in(1,\infty).$ Then
		\[
			\lambda_1(p)\le\liminf_{s\to1^-}\lambda_1(s,p).
		\]	
	\end{cor}

	\begin{proof}
		Let $\{s_k\}_{k\in \N}$ be a sequence in $(0,1)$ such that 
		$s_k\to 1^-$ as $k\to\infty$ and
		$$
			\lim_{k\to \infty} 	\lam_{1}(s_k,p)= \liminf_{s\to 1^-} 	\lam_{1}(s,p).
		$$
		For $k\in\N$, let $u_k$ be the eigenfunctions of problem \eqref{ecu} with $s=s_k$ and $\lam=\lam_1(s_k,p)$ normalized such that 
		$$
			\frac{1}{1-s_k}\int_{\Omega_{1-s_k}}|u_k|^p \, dx =1.
		$$
		Moreover, by Corollary \ref{limsup}, there is a positive
		constant $C$ such that 
		\[
			\|u_k\|_{L^p(\Omega)}\le C \quad \mbox{ and }
			\quad (1-s_k)|u_k|_{W^{s,p}(\Omega)}<C \quad
			\forall k\in\N.
		\]
		Then, by Theorem \ref{BBM}, 
		up to a subsequence, $\{u_k\}$ converges in $\lp$
		(and, in fact, in $W^{s_0,p}(\Omega)$ for all $s_0\in(0,1)$)
		to some $u\in W^{1,p}(\Omega).$ Thus, by Proposition \ref{prop:BPS}
		and Lemma \ref{lema.ricardo}, we get
		\[
		\|\nabla u\|_{\lp}^p\le \lim_{k\to \infty}\mathcal{K}_{n,p}
		(1-s_k)|u_{s_k}|_{W^{s_k,p}(\Omega)}^p
		\]
		and 
		\[
			\lim_{k\to\infty}
			\dfrac{1}{1-s_k}\int_{\Omega_{1-s_k}}|u_k(x)|^p dx
			=\int_{\partial\Omega} |u|^p dS.
		\]
		Then
		$\|u\|_{L^p(\partial\Omega)}^p=1$ and
		\begin{align*}
				\|\nabla u\|_{\lp}^p+\|u\|_{\lp}^p
				&\le  \lim_{k\to \infty}\mathcal{K}_{n,p}
				(1-s_k)|u_{s_k}|_{W^{s_k,p}(\Omega)}^p
				+\|u_k\|_{\lp}^p\\
				&\le\lim_{k\to\infty}\lambda_1(s_k,p)
				=\liminf_{s\to1^-}\lambda_1(s,p).	
		\end{align*}
		Therefore 
		\[
				\lambda_1(p)\le\liminf_{s\to1^-}\lambda_1(s,p).
		\]
	\end{proof}

\end{document}